\documentclass[a4paper,12pt]{scrartcl}

\usepackage[english]{babel} 
\usepackage[utf8]{inputenc} 
\usepackage[T1]{fontenc} 
\usepackage{ae}
\usepackage{textcomp} 
\usepackage{amsfonts}
\usepackage{amsmath}
\usepackage{amssymb}

\usepackage[amsmath,thmmarks,hyperref,amsthm]{ntheorem} 
\usepackage[shortlabels]{enumitem}
\usepackage[pdftex]{hyperref} 
\usepackage{url}
\usepackage[
	backend=biber,
	style=numeric,
	sorting = nyt, 
	maxnames = 8, 
	sortlocale = de_DE,
	firstinits = true, 
	doi=false
]{biblatex}

\newbibmacro{string+doi}[1]{%
	  \iffieldundef{doi}{#1}{\href{http://dx.doi.org/\thefield{doi}}{#1}}}
	  \DeclareFieldFormat{title}{\usebibmacro{string+doi}{\mkbibemph{#1}}}
	  \DeclareFieldFormat[article]{title}{\usebibmacro{string+doi}{\mkbibquote{#1}}}

\usepackage{remreset} 
\makeatletter\@removefromreset{footnote}{chapter}\makeatother

\usepackage{tikz}

\newcommand{\middlerel}[1]{\mathrel{}\middle#1\mathrel{}} 

\newcommand{\F}{\mathbb{F}}

\newcommand{\N}{\mathbb{N}}
\newcommand{\Z}{\mathbb{Z}}

\newcommand{\gauss}[3]{\genfrac{[}{]}{0pt}{}{#1}{#2}_{#3}}
\newcommand{\gaussm}[2]{\genfrac{[}{]}{0pt}{}{#1}{#2}}

\DeclareMathOperator{\PGammaL}{P\Gamma L}

\DeclareMathOperator{\Aut}{Aut}

\DeclareMathOperator{\PG}{PG}

\makeatletter
\def\theorem@checkbold{} 
\renewtheoremstyle{plain}%
{\item[\hskip\labelsep \theorem@headerfont ##1\ ##2 \theorem@separator]}%
{\item[\hskip\labelsep \theorem@headerfont ##1\ ##2\ {\normalfont(##3)}  \theorem@separator]}
\makeatother

\newtheorem{theorem}{Theorem}[section]
\newtheorem{lemma}[theorem]{Lemma}
\newtheorem{fact}[theorem]{Fact}

\theoremstyle{definition}
\newtheorem{definition}[theorem]{Definition}
\theoremstyle{remark}
\newtheorem{remark}[theorem]{Remark}

\author{
Michael Kiermaier
\thanks{
University of Bayreuth, Institute for Mathematics, 95440 Bayreuth, Germany
\newline
email:~\texttt{michael.kiermaier@uni-bayreuth.de}
\newline
homepage:~\url{http://www.mathe2.uni-bayreuth.de/michaelk/}
\newline
Research supported by ESF COST Action IC1104.
}\\
\emph{University of Bayreuth}
\and
Mario Osvin Pavčević
\thanks{
University of Zagreb, Faculty of Electrical Engineering and Computing, Department of Applied Mathematics, Unska 3, 10000 Zagreb, Croatia
\newline
email:~\texttt{mario.pavcevic@fer.hr}
\newline
Research supported by ESF COST Action IC1104.
}\\
\emph{University of Zagreb}
}

\title{Intersection numbers\\for subspace designs}

\begin{document}
\maketitle
\begin{abstract}
Intersection numbers for subspace designs are introduced and $q$-analogs of the Mendelsohn and Köhler equations are given.
As an application, we are able to determine the intersection structure of a putative $q$-analog of the Fano plane for any prime power $q$.
It is shown that its existence implies the existence of a $2$-$(7,3,q^4)_q$ subspace design.
Furthermore, several simplified or alternative proofs concerning intersection numbers of ordinary block designs are discussed.
\end{abstract}

\section{Introduction and preliminaries}
\subsection{History}
The earliest reference for $q$-analogs of block designs (subspace designs) is \cite{Cameron-1974}.
However, the idea is older, since it is stated that ``Several people have observed that the concept of a $t$-design can be generalised [...]''.
They have also been mentioned in a more general context in \cite{Delsarte-1976-JCTSA20[2]:230-243}.
An introduction can be found in \cite[Day~4]{Suzuki-1989-Designs}.

For $q = 2$, the first nontrivial subspace design%
with $t=2$ has been constructed in \cite{Thomas-1987-GeomDed24[2]:237-242} and generalized to arbitrary $q$ in \cite{Suzuki-1990-GrCo6[3]:293-296,Suzuki-1992-GrCo8[4]:381-389}.
The first nontrivial subspace design with $t = 3$ is found in \cite{Braun-Kerber-Laue-2005-DCC34[1]:55-70}.

In \cite[Th.~1.2]{RayChaudhuri-Singhi-1989-LAA114_115:57-68} it has been shown that for fixed parameters $t$, $v$ and $k$ and $\lambda$ sufficiently large, each admissible parameter set $t$-$(v,k,\lambda)$ is realizable as a subspace design with possibly repeated blocks.
In \cite{Fazeli-Lovett-Vardy-2013-arXiv} it has been proven that nontrivial simple subspace designs exist for any value of $t$.

Quite recently \cite{Braun-Etzion-Ostergard-Vardy-Wassermann-2013-arXiv}, a $2$-analog of the Steiner triple system $\operatorname{STS}(13)$ has been found computationally, by applying the Kramer-Mesner method described in \cite{Braun-Kerber-Laue-2005-DCC34[1]:55-70}.
This discovery is a significant breakthrough, since it is
the very first nontrivial $q$-Steiner system with $t > 1$ and refutes the earlier
conjecture that no such $q$-Steiner system exists.

\subsection{Gaussian binomial coefficients}
We define the $q$-analog of a non-negative integer $n$ as
\[
    [n] = [n]_q = \frac{q^n - 1}{q-1}\in\Z[q]
\]
and the $q$-factorial of $n$ as
\[
    [n]! = [n]_q! = \prod_{i=1}^n [i]\in\Z[q]\text{.}
\]
The notion \emph{$q$-analog} stems from the fact that the evaluation for $q = 1$ gives $[n]_1 = n$ and $[n]_1! = n!$.
Using this notation, for $k\in\Z$ and $n\in\N$ the \emph{Gaussian binomial coefficient} is the $\Z[q]$-polynomial
\[
\gaussm{n}{k} =
\gauss{n}{k}{q} =
\begin{cases}
\frac{[n]!}{[k]!\cdot [n-k]!} & \text{if }k\in\{0,\ldots,n\}\text{,} \\
0 & \text{otherwise.}
\end{cases}
\]
Its evaluation for $q = 1$ gives the binomial coefficient $\binom{n}{k}$.
For that reason, the Gaussian binomial coefficient is known as $q$-analog of the binomial coefficient.
Many identities for binomial coefficients have $q$-analogs for the Gaussian binomial coefficients.
As an example, we mention
\[
	\gaussm{n}{k} = \gaussm{n}{n - k}
	\qquad\text{and}\qquad
	\gaussm{n}{h}\gaussm{n-h}{k} = \gaussm{n}{k}\gaussm{n-k}{h}\text{,}
\]
for $n\geq 1$ the $q$-Pascal triangle identities
\begin{equation*}
	\gauss{n}{k}{q} = \gauss{n-1}{k-1}{q} + q^k\gauss{n-1}{k}{q} = q^{n-k}\gauss{n-1}{k-1}{q} + \gauss{n-1}{k}{q}
\end{equation*}
and the identity (see for example \cite[Eq.~(3)]{Delsarte-1976-JCTSA20[2]:230-243})
\begin{equation}
	\label{eq:formula_inverse_pascal_matrix}
	\sum_{i=a}^b (-1)^{i-a} q^{\binom{i-a}{2}} \gauss{b}{i}{q} \gauss{i}{a}{q}
	= \begin{cases}
		1 & \text{if }a=b\text{,} \\
		0 & \text{else.}
	\end{cases}
\end{equation}
Another one is provided in the following
\begin{lemma}
	\label{lma:alternating_gauss_sum}
	Let $n,k\in\Z$ with $n\geq 1$.
	Then
\begin{equation*}
	\sum_{i=0}^k (-1)^i q^{\binom{i}{2}} \gauss{n}{i}{q} = (-1)^k q^{\binom{k+1}{2}} \gauss{n-1}{k}{q}\text{.}
\end{equation*}
\end{lemma}

\begin{proof}
    By the $q$-Pascal triangle identity,
    \begin{align*}
        \sum_{i = 0}^k (-1)^i q^{\binom{i}{2}}\gauss{n}{i}{q}
	& = \sum_{i = 0}^k (-1)^i q^{\binom{i}{2}} \cdot q^i \gauss{n-1}{i}{q} + \sum_{i = 0}^k (-1)^i q^{\binom{i}{2}} \cdot \gauss{n-1}{i-1}{q} \\
	& = \sum_{i = 0}^k (-1)^i q^{\binom{i+1}{2}} \gauss{n-1}{i}{q} - \sum_{i=-1}^{k-1} (-1)^i q^{\binom{i+1}{2}} \gauss{n-1}{i}{q} \\
	& = (-1)^k q^{\binom{k+1}{2}}\gauss{n-1}{k}{q}\text{.}
    \end{align*}
\end{proof}

\begin{remark}
According to \cite{Fray-1967-DukeMJ34[3]:467-480}, $q$-analogs of non-negative integers were introduced in \cite{Jackson-1910-AJM32[3]:305-314} and their binomial coefficients in \cite{Ward-1936-AJM58[2]:255-266}.
For a deeper discussion of the Gaussian binomial coefficients, see \cite{Polya-Alexanderson-1971-ElemMath26:102-109,Fray-1967-DukeMJ34[3]:467-480,Goldman-Rota-1970-StudApplM49[3]:239-258,Cohn-2004-AMM111[6]:487-495}.
\end{remark}

\subsection{$q$-analogs of combinatorial structures}
The set of $k$-element subsets (\emph{$k$-subsets}) of a set $V$ will be denoted by $\binom{V}{k}$ and the set of all $k$-dimensional subspaces (\emph{$k$-subspaces}) of an $\F_q$-vector space will be denoted by $\gaussm{V}{k}$ or $\gauss{V}{k}{q}$.
The latter is known as the \emph{Graßmannian}.
This notation is chosen for the fact that the size of $\binom{V}{k}$ equals the binomial coefficient $\binom{\#V}{k}$, and the size of $\gauss{V}{k}{q}$ equals the Gaussian binomial coefficient $\gauss{\dim(V)}{k}{q}$.
There are good reasons to interpret the subspace lattice $\mathcal{L}(V)$ of a $v$-dimensional vector space $V$ as the $q$-analog of the subset lattice of a $v$-element set $V$, which corresponds to $q = 1$ \cite{Cohn-2004-AMM111[6]:487-495}.

Many combinatorial areas, like design theory and coding theory, are based on the subset lattice of a $v$-element set $V$.
Replacing the set-theoretic notions by their vector space counterparts gives rise to the study of their $q$-analogs, which are based on the subspace lattice of a $v$-dimensional vector space $V$.
An important part of these theories is the investigation of results in the set-theoretic case for their applicability in the $q$-analog case.
For example, in \cite{Kiermaier-Laue-2014-arXiv} $q$-analogs of derived and residual designs are studied.
In this article, we will give a $q$-analog of the theory of intersection numbers of designs.

In the following, $q$-analogs of several well-known definitions and statements on ordinary block designs are given (Def.~\ref{def:design}, Fact~\ref{fct:integrality}, Fact~\ref{fct:lambda_double}, Fact~\ref{fct:dual}, Def.~\ref{def:intersection_number}, Th.~\ref{thm:mendelson_equations}, Th.~\ref{thm:q_koehler}).
This means that one gets back the original definition or statement if $q$ is set to $1$ and all vector space notions are replaced by their set-theoretic counterparts.%

\subsection{Subspace designs}
\begin{definition}
\label{def:design}
Let $q$ be a prime power, $V$ an $\F_q$-vector space of finite dimension $v$ and $t,k,\lambda$ be non-negative integers.
A set $D$ of $k$-subspaces (\emph{blocks}) of $V$ is called a \emph{$t$-$(v,k,\lambda)_q$ (subspace) design} if each $t$-subspace of $V$ is contained in exactly $\lambda$ blocks of $D$.
\end{definition}

By the above discussion, an ordinary block design can be seen as the case $q = 1$ of a subspace design.
For all $t\in\{0,\ldots,k\}$, the full Graßmannian $\gaussm{V}{k}$ forms the \emph{trivial} $t$-$(v,k,\gauss{v-t}{k-t}{q})_q$ subspace design.
It is clear that for any $t$-$(v,k,\lambda)_q$ subspace design $D$, the \emph{complementary} design $\gaussm{V}{k}\setminus D$ is a $t$-$(v,k,\gauss{v-t}{k-t}{q} - \lambda)_q$ subspace design.

Of particular interest is the case $\lambda = 1$, where $D$ is called a \emph{Steiner system}.
For $t = 1$, a $1$-$(v,k,1)_q$ Steiner system is the same as a spread of $(k-1)$-flats in the projective geometry $\PG(v-1,q)$, which exists if and only if $k$ divides $v$.
The only known nontrivial $q$-analog of a Steiner system with $t \geq 2$ has the parameters $2$-$(13,3,1)_2$ \cite{Braun-Etzion-Ostergard-Vardy-Wassermann-2013-arXiv}.

By the fundamental theorem of projective geometry, the automorphism group of the lattice $\mathcal{L}(V)$ is given by the projective semilinear group $\PGammaL(V)$ with its natural action on $\mathcal{L}(V)$.
The automorphism group $\Aut(D)$ of a subspace design $D$ is defined as the stabilizer of $D$ under the induced action of $\PGammaL(V)$ on the power set of $\mathcal{L}(V)$.

The following fact is the $q$-analog of a well-known property of block designs:
\begin{fact}[{{\cite[Lemma~4.1(1)]{Suzuki-1990-EuJC11[6]:601-607}}}]
	\label{fct:integrality}
	Let $D$ be $t$-$(v,k,\lambda)_q$ subspace design.
	For each $i \in\{0,\ldots,t\}$, $D$ is an $i$-$(v,k,\lambda_i)_q$ subspace design with
	\[
		\lambda_i
		= \frac{\gauss{v-i}{t-i}{q}}{\gauss{k-i}{t-i}{q}}\cdot\lambda
		= \frac{\gauss{v-i}{k-i}{q}}{\gauss{v-t}{k-t}{q}}\cdot\lambda\text{.}
	\]
	In particular, the number of blocks is given by $\#D = \lambda_0$.
\end{fact}

As a result, the existence of a $t$-$(v,k,\lambda)_q$ design implies the \emph{integrality conditions} $\lambda_i\in\Z$ for all $i\in\{0,\ldots,t\}$.
Without requiring the actual existence of a corresponding design, any parameter set $t$-$(v,k,\lambda)_q$ fulfilling the integrality conditions will be called \emph{admissible}.

The following fact describes a refinement of the numbers $\lambda_i$.
\begin{fact}[{{\cite[Lemma~2.1]{Suzuki-1990-EuJC11[6]:601-607}, \cite[Lemma~4.2]{Suzuki-1989-Designs}}}]
	\label{fct:lambda_double}
	Let $D$ be a $t$-$(v,k,\lambda)_q$ subspace design and $i,j$ non-negative integers with $i + j \leq t$.
	Let $I \in \gaussm{V}{i}$ and $J\in\gaussm{V}{j}$ with $I\cap J = \{\mathbf{0}\}$.
	The number
	\[
	    \lambda_{i,j} = \{ B\in D \mid I\leq B\text{ and }B\cap J = \{\mathbf{0}\}\}
	\]
	is independent of the choice of $I$ and $J$.
	They are determined by the recurrence relation
	\[
		\lambda_{i,0} = \lambda_i
		\qquad\text{and}\qquad
		\lambda_{i,j+1} = \lambda_{i,j} - q^j\lambda_{i+1,j}\text{.}
	\]
	In closed form,
	\[
		\lambda_{i,j} = q^{j(k-i)} \frac{\gauss{v-i-j}{k-i}{q}}{\gauss{v-t}{k-t}{q}}\cdot\lambda\text{.}
	\]
\end{fact}

\begin{remark}
\begin{enumerate}[(i)]
\item Fact~\ref{fct:lambda_double} can also be found in \cite[Lemma~2.2 and Lemma~2.6]{Suzuki-1990-HokkMS19[3]:403-415}.
However, the exponents of $q$ given there are not correct.
\item For ordinary block designs, the numbers $\lambda_{i,j}$ have been introduced in~\cite{RayChaudhuri-Wilson-1975-OsakaJM12[3]:737-744}.
\end{enumerate}
\end{remark}

Fixing some non-singular bilinear form $\beta$ on $V$, the \emph{dual subspace} of a subspace $W\in\mathcal{L}(V)$ is defined as
\[
    W^\perp = \{x\in V \mid \beta(x,y) = 0 \text{ for all }y\in W\}\text{.}
\]
Now for a $t$-$(v,k,\lambda)_q$ subspace design $D$, its \emph{dual} subspace design is defined as
\[
    D^\perp = \{B^\perp \mid B\in \mathcal{B}\}\text{.}
\]
Up to equivalence, this definition does not depend on the choice of $\beta$.
The dual subspace design is the $q$-analog of the supplementary block design.

\begin{fact}[{{\cite[Lemma~4.2]{Suzuki-1990-EuJC11[6]:601-607}}}]
\label{fct:dual}
Let $D$ be a $t$-$(v,k,\lambda)_q$ subspace design.
Then $D^\perp$ is a subspace design with the parameters
\[
	t\text{-}\left(v,v-k,\frac{\gauss{v-t}{k}{q}}{\gauss{v-t}{k-t}{q}}\cdot\lambda\right)_{\!\!q}\text{.}
\]
\end{fact}

\section{Intersection numbers}
\subsection{Plain intersection numbers}
\begin{definition}
\label{def:intersection_number}
Let $D$ be a $t$-$(v,k,\lambda)_q$ subspace design.
For any subspace $S$ of $V$ and $i\in\{0,\ldots,k\}$, we define the $i$-th intersection number of $S$ in $D$ as
\[
	\alpha_i(S) = \#\left\{B\in D \mid \dim(B\cap S) = i\right\}\text{.}
\]
If the set $S$ is clear from the context, we use the abbreviation $\alpha_i = \alpha_i(S)$.
Furthermore, the $(k+1)$-tuple $\alpha(S) = (\alpha_0(S),\alpha_1(S),\ldots,\alpha_k(S))$ will be called the \emph{intersection vector} of $S$ in $D$.
\end{definition}

The intersection numbers are a $q$-analog of the intersection numbers defined in \cite{Mendelsohn-1971} for blocks $S$ and independently as ``$i$-Treffer'' for general sets $S$ in \cite{Oberschelp-1972-MPSemBNF19:55-67}.

First, we describe the relation to the intersection numbers $\alpha_i^\perp$ of the dual design $D^\perp$:

\begin{lemma}
	\label{lma:alpha_dual}
    Let $D$ be a $t$-$(v,k,\lambda)_q$ subspace design, $S\in\mathcal{L}(V)$, $s = \dim(S)$ and $i\in\{0,\ldots,k\}$.
    \begin{enumerate}[(a)]
    	\item\label{lma:alpha_dual:zero} For $i > s$ or $k-i > v-s$ we have
	\[
	    \alpha_i(S) = 0\text{.}
	\]
	\item\label{lma:alpha_dual:dual} 
	     For $i \leq s$ and $k-i \leq v-s$ we have
	    \[
		\alpha_i(S) = \alpha^{\perp}_{(v-s) - (k-i)}(S^\perp)\text{.}
	    \]
    \end{enumerate}
\end{lemma}

\begin{proof}
Part~\ref{lma:alpha_dual:zero} is a direct consequence of the dimension formula.
For part~\ref{lma:alpha_dual:dual}, note that for all blocks $B\in D$
\begin{align*}
    \dim(B\cap S) = i 
    & \iff \dim(B + S) = k + s - i \\
    & \iff \dim(B^\perp \cap S^\perp) = v - (k + s - i)\text{.}
\end{align*}
\end{proof}

In the range where the dimension or the codimension of $S$ in $V$ is at most $t$, the intersection numbers are closely related to the numbers $\lambda_{i,j}$:

\begin{lemma}
\label{lma:alpha_unique}
Let $D$ be a $t$-$(v,k,\lambda)_q$ subspace design, $S$ a subspace of $V$ of dimension
\[
    s = \dim(S) \in\{0,\ldots,t\}\cup\{v-t,\ldots,v\}
\]
and $i\in\{0,\ldots,k\}$.
The intersection vector $\alpha(S)$ is uniquely determined by
\begin{equation}
    \label{eq:alpha_unique}
    \alpha_i(S) 
    = q^{(s-i)(k-i)} \frac{\gauss{s}{i}{q}\gauss{v-s}{k-i}{q}}{\gauss{v-t}{k-t}{q}}\cdot\lambda\text{.}
\end{equation}
For $i \leq s \leq t$ we have
\begin{equation}
    \label{eq:alpha_unique_small}
    \alpha_i(S) = \gauss{s}{i}{q}\cdot\lambda_{i,s-i}
\end{equation}
and for $k-i \leq v-s \leq t$ we have
\begin{equation}
    \label{eq:alpha_unique_large}
    \alpha_i(S) = q^{sk-iv}\gauss{v-s}{k-i}{q}\cdot\lambda_{k-i,(v-s) - (k-i)} \text{.}
\end{equation}
\end{lemma}

\begin{proof}
From Lemma~\ref{lma:alpha_dual}\ref{lma:alpha_dual:zero}, $\alpha_i(S) = 0$ for $i > s$ or $k-i > v-s$, in agreement with equation~\eqref{eq:alpha_unique}.
So we may assume $i \leq s$ and $k-i \leq v-s$.

\paragraph{Case 1}
We first consider the case $s\leq t$.
We count the set
\[
    X = \left\{(B,I) \in D \times \gauss{S}{i}{q} \middlerel{|} B \cap S = I\right\}
\]
in two ways.

There are $\alpha_i(S)$ blocks $B$ with $\dim(B\cap S) = i$, each one uniquely determining $I = B\cap S$.
This shows that $\#X$ equals the left hand side of equation~\eqref{eq:alpha_unique_small}.

On the other hand, there are $\gauss{s}{i}{q}$ ways to select the subspace $I$ of $S$.
For fixed $I$, let $J$ be a complement of $I$ in $S$.
Let $B\in D$ with $I \leq B$.
If $B\cap S \neq I$, then $\dim(B \cap S) > i$, and because of $\dim(J) + \dim(B \cap S) > (s-i) + i = s$, we get that $\dim(J \cap B) > 1$.
Hence $B \cap S = I$ is equivalent to $I \leq B$ and $J\cap B = \{\mathbf{0}\}$.
So the number of blocks intersecting $S$ in $I$ is $\lambda_{\dim(I),\dim(J)} = \lambda_{i,s-i}$, showing that $\#X$ equals the right hand side of equation~\eqref{eq:alpha_unique_small}.

So equation~\eqref{eq:alpha_unique_small} is shown, and replacing $\lambda_{i,s-i}$ with the formula given in Fact~\ref{fct:lambda_double} yields formula~\eqref{eq:alpha_unique}.

\paragraph{Case 2}
Now assume that $v-s \leq t$.
By Fact~\ref{fct:dual}, the dual design $D^\perp$ has the parameters $t^\perp = t$, $v^\perp = v$, $k^\perp = v-k$ and $\lambda^\perp = \gauss{v-t}{k}{q}/\gauss{v-t}{k-t}{q}\cdot\lambda$.
We further define $s^\perp = v-s$ and $i^\perp = (v-s)-(k-i)$.
Now
\[
\alpha_i(S)
= \alpha^\perp_{i^\perp}(S^\perp)
= q^{(s^\perp-i^\perp)(k^\perp-i^\perp)} \frac{\gauss{s^\perp}{i^\perp}{q}\gauss{v^\perp-s^\perp}{k^\perp-i^\perp}{q}}{\gauss{v^\perp-t^\perp}{k^\perp-t^\perp}{q}}\cdot\lambda^\perp\text{,}
\]
where the first equality is Lemma~\ref{lma:alpha_dual}\ref{lma:alpha_dual:dual}, and the second equality comes from applying equation~\eqref{eq:alpha_unique} (because of $\dim(S^\perp) = s^\perp \leq t^\perp$ we are in case~1 that we have already shown).
Plugging in the above defined expressions, this expression indeed simplifies to the right hand side of~\eqref{eq:alpha_unique}.
Finally equation~\eqref{eq:alpha_unique_large} can be verified using the formula from Fact~\ref{fct:lambda_double}.
\end{proof}

\begin{theorem}[{{$q$-analog of the Mendelsohn equations \cite[Th.~1]{Mendelsohn-1971}}}]
\label{thm:mendelson_equations}
Let $D$ be a $t$-$(v,k,\lambda)_q$ subspace design, $S$ a subspace of $V$ and $s = \dim(S)$.
For $i\in\{0,\ldots,t\}$ we have the following equation on the intersection numbers of $S$ in $D$:
\begin{equation}
\label{eq:mendelsohn}
	\sum_{j = i}^s \gauss{j}{i}{q} \alpha_j = \gauss{s}{i}{q}\lambda_i
\end{equation}
\end{theorem}

\begin{proof}
We count the set
\[
X = \left\{ (I,B)\in \gauss{V}{i}{q} \times D \middlerel{|} I \leq B\cap S\right\}
\]
in two ways.

There are $\gauss{s}{i}{q}$ possibilities for the choice of $I\in\gaussm{S}{i}$.
By Fact~\ref{fct:integrality}, there are $\lambda_i$ blocks $B$ such that $I \leq B$, which shows that $\#X$ equals the right hand side of equation~\eqref{eq:mendelsohn}.

Fixing a block $B$, the number of $i$-subspaces $I$ of $B\cap S$ is $\gauss{\dim(B\cap S)}{i}{q}$.
Summing over the possibilities for $j = \dim(B\cap S)$, we see that $\#X$ also equals the left hand side of equation~\eqref{eq:mendelsohn}.
\end{proof}

\begin{remark}
In \cite[Th.~1]{Mendelsohn-1971}, $S$ was required to be a block.
For general $S$, the equations were given independently in \cite[Satz~2]{Oberschelp-1972-MPSemBNF19:55-67}.
\end{remark}

\begin{theorem}[{{$q$-analog of the Köhler equations \cite[Satz~1]{Koehler-1988_1989-DM73[1-2]:133-142}}}]
\label{thm:q_koehler}
Let $D$ be a $t$-$(v,k,\lambda)_q$ subspace design, $S$ a subspace of $V$ and $s = \dim(S)$.
For $i\in\{0,\ldots,t\}$, a parametrization of the intersection number $\alpha_i$ by $\alpha_{t+1},\ldots,\alpha_k$ is given by
\begin{multline*}
	\alpha_i
	= \gauss{s}{i}{q}\sum_{j = i}^t (-1)^{j-i} q^{\binom{j-i}{2}} \gauss{s-i}{j-i}{q}\lambda_j\\
	+ (-1)^{t+1-i} q^{\binom{t+1-i}{2}} \sum_{j = t+1}^k \gauss{j}{i}{q} \gauss{j-i-1}{t-i}{q}\alpha_j\text{.}
\end{multline*} 
\end{theorem}

\begin{proof}
The Mendelsohn equations (Th.~\ref{thm:mendelson_equations}) can be interpreted as a system of linear equations on the intersection vector of $S$ in $D$:
		\[
			\left(\begin{array}{ccccc|ccc}
				\gauss{0}{0}{q} & \gauss{1}{0}{q} & \gauss{2}{0}{q} & \hdots & \gauss{t}{0}{q} & \gauss{t+1}{0}{q} & \hdots & \gauss{k}{0}{q} \\
				0 & \gauss{1}{1}{q} & \gauss{2}{1}{q} & \hdots & \gauss{t}{1}{q} & \gauss{t+1}{1}{q} & \hdots  & \gauss{k}{1}{q} \\
				0 & 0 & \gauss{2}{2}{q} & \hdots & \gauss{t}{2}{q} & \gauss{t+1}{2}{q} & \hdots & \gauss{k}{2}{q} \\
				 \vdots & & \ddots & \ddots & \vdots & \vdots & & \vdots \\
				 0 & 0 & \hdots & 0 & \gauss{t}{t}{q} & \gauss{t+1}{t}{q} & \hdots & \gauss{k}{t}{q}
			\end{array}\right)
			 \begin{pmatrix}
				\alpha_0 \\ \alpha_1 \\ \alpha_2 \\ \vdots \\ \alpha_k
			 \end{pmatrix}
			 =
			 \begin{pmatrix}
			 	\gauss{s}{0}{q} \lambda_0 \\
			 	\gauss{s}{1}{q} \lambda_1 \\
			 	\gauss{s}{2}{q} \lambda_2 \\
				\vdots \\
			 	\gauss{s}{t}{q} \lambda_t \\
			 \end{pmatrix}
		\]

		This equation system has the form
		\begin{equation}
			\label{eq:mendelsohn_linear}
			(P_q \mid A) \cdot \mathbf{x} = \mathbf{b}
		\end{equation}
		where
		\begin{align*}
		    P_q & = \left(\gauss{j}{i}{q}\right)_{i,j\in\{0,\ldots,t\}}\in\Z^{(t+1)\times(t+1)}\text{,} \\
		    A & = \left(\gauss{j}{i}{q}\right)_{i\in\{0,\ldots,t\},j\in\{t+1,\ldots,k\}}\in\Z^{(t+1)\times (k-t)}\quad\text{and} \\
		    \mathbf{b} & = \left(\gauss{s}{i}{q}\lambda_i\right)_{i\in\{0,\ldots,t\}}\in\Z^{t+1}\text{.}
		\end{align*}
		The matrix $P_q$ is known as the \emph{upper triangular $q$-Pascal matrix}.
		By equation~\eqref{eq:formula_inverse_pascal_matrix}, $P_q$ is invertible with the inverse
		\[
			P_q^{-1} = \left((-1)^{j-i}q^{\binom{j-i}{2}}\gauss{j}{i}{q}\right)_{i,j\in\{0,\ldots,t\}}\text{.}
		\]

		After left multiplication by $P_q^{-1}$, equation~\eqref{eq:mendelsohn_linear} is equivalent to
		\begin{equation}
			\label{eq:mendelsohn_linear_reduced}
			(I \mid P_q^{-1}A)\cdot \mathbf{x} = P_q^{-1}\mathbf{b}\text{.}
		\end{equation}
		Numbering the columns of $A$ with $t+1,\ldots,k$, the entry in the $i$-th row and the $j$-th column of $P_q^{-1}A$ is
		\begin{align*}
			(P_q^{-1}A)_{i,j}
			& = \sum_{\nu = 0}^t (-1)^{\nu-i}q^{\binom{\nu-i}{2}}\gauss{\nu}{i}{q} \gauss{j}{\nu}{q} \\
			& = \sum_{\nu = i}^t (-1)^{\nu-i}q^{\binom{\nu-i}{2}}\gauss{j}{i}{q} \gauss{j-i}{\nu-i}{q} \\
			& = \gauss{j}{i}{q}\sum_{\nu = 0}^{t-i} (-1)^\nu q^{\binom{\nu}{2}} \gauss{j-i}{\nu}{q} \\
			& = \gauss{j}{i}{q} (-1)^{t-i} q^{\binom{t-i+1}{2}}\gauss{j-i-1}{t-i}{q}\text{,}
		\end{align*}
		where Lemma~\ref{lma:alternating_gauss_sum} was used in the last step.
		The $i$-th entry of $P_q^{-1}\mathbf{b}$ is
		\begin{align*}
			(P_q^{-1}\mathbf{b})_i
			& = \sum_{j = 0}^t (-1)^{j - i} q^{\binom{j - i}{2}} \gauss{j}{i}{q}\gauss{s}{j}{q}\lambda_j \\
			& = \gauss{s}{i}{q}\sum_{j = i}^t (-1)^{j - i} q^{\binom{j-i}{2}} \gauss{s-i}{j-i}{q}\lambda_j\text{.}
		\end{align*}
		Plugging these expressions into equation~\eqref{eq:mendelsohn_linear_reduced}, its rows evaluate to the Köhler equations.
	\end{proof}

\begin{remark}
	\begin{enumerate}[(a)]
		\item For ordinary block designs, Theorem~\ref{thm:q_koehler} was originally shown in \cite{Koehler-1988_1989-DM73[1-2]:133-142} in a lengthy induction proof.
The main result of the article \cite{Vroedt-1991-DM97[1-3]:161-165} was a simplified proof based on the notion of ``vectorproduct''.
In a slightly more general context, another induction proof as well as a proof based on the principle of inclusion and exclusion was given in \cite{Trung-Wu-Mesner-1996-JSPI56[2]:257-268}.

		\item Our proof can be interpreted as transforming the linear system of Mendelsohn equations to row reduced echelon form by Gauss reduction.
		Since this method is directly applicable also to block designs, it provides a short and systematic proof for the original Köhler equations.
	\end{enumerate}
\end{remark}

\subsection{High order intersection numbers}
For block designs, ``high order'' versions of the numbers $\lambda_{i,j}$ \cite{Trung-Wu-Mesner-1996-JSPI56[2]:257-268} and the intersection numbers $\alpha_i$ \cite{Mendelsohn-1971} (see also \cite{Trung-Wu-Mesner-1996-JSPI56[2]:257-268, Betten-Dissertation-1998}) have been introduced.
For that matter, some positive integer $\ell$ is fixed, and in the definitions the block $B$ is replaced by the intersection of an $\ell$-tuple of blocks.

The same is possible in our $q$-analog situation:
For a $t$-$(v,k,\lambda)_q$ subspace design $D$, non-negative integers $i$ and $j$ with $i + j \leq t$ and subspaces $I\in\gaussm{V}{i}$ and $J\in\gaussm{V}{j}$ with $I\cap J = \{\mathbf{0}\}$, the number
\[
	\lambda^{(\ell)}_{i,j} = \#\left\{\mathcal{B}\in\binom{D}{\ell} \middlerel{|} I \leq \bigcap\mathcal{B}\text{ and }\bigcap\mathcal{B}\cap J = \{\mathbf{0}\}\right\}
\]
does not depend on the choice of $I$ and $J$.
For $S\in\mathcal{L}(V)$ and $i\in\{0,\ldots,k\}$ the $i$-th \emph{high order intersection number} of $S$ in $D$ is defined as
\[
	\alpha^{(\ell)}_i(S) = \#\left\{\mathcal{B}\in\binom{D}{\ell} \middlerel{|} \dim\left(\bigcap\mathcal{B} \cap S\right) = i\right\}\text{.}
\]
Clearly, $\lambda_{i,j} = \lambda^{(1)}_{i,j}$ and $\alpha_i(S) = \alpha^{(1)}_i(S)$, so the high order versions generalize the basic versions considered so far.

Replacing $\alpha_i$ by $\alpha_i^{(\ell)}$, $\lambda_{i,j}$ by $\lambda^{(\ell)}_{i,j}$ and $\lambda_i$ by $\binom{\lambda_i}{\ell}$, we get high order versions of the statements of Fact~\ref{fct:lambda_double} (except the closed formula), Th.~\ref{thm:mendelson_equations} and Th.~\ref{thm:q_koehler}.
A partial high order version of Lemma~\ref{lma:alpha_unique} is the following:
\begin{lemma}
	Let $D$ be a $t$-$(v,k,\lambda)_q$ subspace design, $S$ a subspace of $V$ of dimension $s = \dim(S)$ and $\ell$ a positive integer.
	For $s \leq t$ or $s\geq v-t$, the high order intersection vector $\alpha^{(\ell)}(S)$ is uniquely determined.
	In the case $s \leq t$,
	\[
		\alpha^{(\ell)}_i(S)
		= \begin{cases}
		    \gauss{s}{i}{q} \cdot\lambda^{(\ell)}_{i,s-i} & \text{for all }i\in\{0,\ldots,s\} \text{ and} \\
		    0 & \text{for all }i\in\{s+1,\ldots,k\}\text{.}
		\end{cases}
	\]
\end{lemma}
Using the high order Mendelsohn equations, it can be checked that in the range $s\geq v-t$, the intersection vector $\alpha^{(\ell)}_i(S)$ is still unique.
However, the formula gets more complicated than in Lemma~\ref{lma:alpha_unique}.
This is indicated by the fact that for $k-i > v-s$, we don't necessarily get $\alpha_i^{(\ell)}(S) = 0$ any more.

Since the high order versions complicate the presentation, their benefit is not entirely clear and the proofs only need trivial adjustments, we decided to go with the basic versions in the main part.

\section{Non-existence results for block designs}
For ordinary block designs, the Mendelsohn equations have been used to show that certain admissible parameter sets are not realizable.
Below, we give three such examples.
These results are not new, but the proofs are new alternatives or simplify the previous ones.

\begin{theorem}
\label{thm:3_11_5_2}
The parameter set $3$-$(11,5,2)$ is admissible, but not realizable.
\end{theorem}

\begin{proof}
The numbers
\[
	\lambda_0 = 33,\quad \lambda_1 = 15,\quad \lambda_2 = 5,\quad \lambda_3 = \lambda = 2
\]
are all integers, so the parameter set is admissible.

To show that the parameter set is not realizable, let $V = \{1,\ldots,11\}$ and assume that there is a design on $V$ of these parameters.
The Köhler equations for the intersection vector of a block $B$ are
\begin{align*}
	\alpha_0 & = -2 + \alpha_4 + 4\alpha_5\text{,} &
	\alpha_1 & = 15 - 4\alpha_4 - 15\alpha_5\text{,} \\
	\alpha_2 & = 6\alpha_4 + 20\alpha_5\text{,} &
	\alpha_3 & = 20 - 4\alpha_4 - 10\alpha_5\text{.}
\end{align*}
Since $B$ is a block, $\alpha_5 = 1$, and because of $\alpha_1 \geq 0$, the second equation forces $\alpha_4 = 0$.
So the unique intersection vector is
\[
	\alpha(B) = (2,0,20,10,0,1)\text{.}
\]

In particular, there are $\alpha_0 = 2$ blocks contained in $V\setminus B$.
Because of $\#(V\setminus B) = 11 - 5 = 6$, those two blocks intersect in exactly $4$ points, which contradicts $\alpha_4 = 0$.
\end{proof}

\begin{remark}
\begin{enumerate}[(i)]
\item For block designs, the considered parameter set $2$-$(11,5,2)$ is the smallest admissible parameter set (in terms of $v$) which is not realizable, compare \cite[p.~36 ff.]{Mathon-Rosa-2designs-2007} and \cite[Table~4.44]{Khosrovshahi-Laue-2007}.
\item Theorem~\ref{thm:3_11_5_2} is the main result of \cite{Dehon-1976-DM15[1]:23-25}, where it was shown using the same intersection vector and additionally the classification of $2$-$(10,4,2)$ designs.
Our above proof simplifies this reasoning.
\end{enumerate}
\end{remark}

\begin{theorem}
    \label{thm:family_t_4}
    Let $n \geq 5$ be an integer such that $4\nmid n$.
    Then the parameters
    \[
	t = 4\text{,}\quad v = \binom{n}{2} + 2\text{,}\quad k = n+1\text{,}\quad \lambda = 2
    \]
    are admissible, but not realizable.
\end{theorem}

\begin{proof}
    We compute
    \begin{align*}
        \lambda_3 &
	= \lambda_4\cdot \frac{v-3}{k-3}
	= 2\cdot \frac{\frac{n(n-1)}{2} - 1}{n-2}
	= n + 1\text{,} \\
	\lambda_2 & = \lambda_3 \cdot \frac{v-2}{k-2}
	= \frac{(n+1) \frac{n(n-1)}{2}}{n-1} = \frac{n(n+1)}{2}\text{,} \\
	\lambda_1 & = \lambda_2 \cdot \frac{v-1}{k-1}
	= \frac{n(n+1)}{2}\cdot \frac{\frac{n(n-1)}{2} + 1}{n} = \frac{(n+1)(n^2 - n + 2)}{4}\text{,} \\
	\lambda_0 & = \lambda_1 \cdot \frac{v}{k} = \frac{(n+1)(n^2 - n + 2)}{4} \cdot \frac{\frac{n(n-1)}{2} + 2}{n+1} = \frac{(n^2 - n + 2)(n^2 - n + 4)}{8}\text{.}
    \end{align*}
    To see that the parameters are admissible, we have to check that the values $\lambda_i$ are integral.
    This is clear for $\lambda_4$, $\lambda_3$ and $\lambda_2$.
    The integrality of $\lambda_1$ follows from checking the three possibilities $n\equiv 1$, $n\equiv 2$ and $n\equiv 3\pmod 4$.
    For $\lambda_0$ we note that $n^2 - n$ is always even, so one of the factors $n^2 - n + 2$ and $n^2 - n + 4$ is divisible by $4$ and the other one is even.

    We consider the Köhler equation with $i=0$ for a block $B$ (so $s = k = n+1$ and $\alpha_k = 1$).
    Because of
    \begin{multline*}
    \binom{s}{i}\sum_{j = i}^t (-1)^{j-i} \binom{s-i}{j-i} \lambda_j \\
    = \frac{(n^2 - n + 2) (n^2 - n + 4)}{8}
    - \frac{(n+1)^2(n^2 - n + 2)}{4}
    + \frac{n^2(n+1)^2}{4} \\
    - \frac{(n-1)n(n+1)^2}{6}
    + \frac{(n-2)(n-1)n(n+1)}{12} \\
    = \frac{(n-1)(n-2)^2(n-3)}{24}\text{,}
    \end{multline*}
    we get the contradiction
    \begin{align*}
    	\alpha_0 & = \frac{(n-1)(n-2)^2(n-3)}{24} - \sum_{j=5}^n \binom{j-1}{4}\alpha_j - \binom{n}{4} \\
	& = -\frac{(n-1)(n-2)(n-3)}{12} - \sum_{j=5}^n \binom{j-1}{4}\alpha_j < 0\text{.}
    \end{align*}
\end{proof}

\begin{remark}
    Designs with the parameters from Theorem~\ref{thm:family_t_4} would be tight, since $t = 4 = 2s$ with $s = 2$ and $\lambda_0 = \binom{v}{s}$.
    The existence of this series was ruled out in \cite[Cor. of Th.~5]{RayChaudhuri-Wilson-1975-OsakaJM12[3]:737-744}, the parameters are explicitly stated as $S_2\!\left(4,k,2+\frac{1}{2}(k-1)(k-2)\right)$ in \cite[p.~738]{RayChaudhuri-Wilson-1975-OsakaJM12[3]:737-744}.
\end{remark}

\begin{theorem}
    \label{thm:family_t_3}
    Let $n \geq 2$ be an integer.
    Then the parameters
    \[
	t = 3\text{,}\quad v = (2n-1)(4n-1) + 1\text{,}\quad k = 4n-1\text{,}\quad \lambda = 1
    \]
    are admissible, but not realizable.
\end{theorem}

\begin{proof}
    We compute
    \begin{align*}
    	\lambda_2
	& = \lambda_3\cdot \frac{v-2}{k-2}
	= \frac{(2n-1)(4n-1) - 1}{4n-3}
	= 2n\text{,} \\
	\lambda_1
	& = \lambda_2 \cdot \frac{v-1}{k-1}
	= 2n\cdot \frac{(2n-1)(4n-1)}{4n-2}
	= n(4n - 1)\text{,} \\
	\lambda_0
	& = \lambda_1 \cdot \frac{v}{k}
	= n(4n - 1)\frac{(2n-1)(4n-1) + 1}{4n-1} = 2n(4n^2 - 3n + 1)\text{.}
    \end{align*}
    So the parameters are admissible.

    We consider the Köhler equation with $i = 1$ for a block $S$ (so $s = k = 4n-1$ and $\alpha_k = 1$).
    Because of
    \begin{multline*}
	\binom{s}{i}\sum_{j = i}^t (-1)^{j-i} \binom{s-i}{j-i} \lambda_j \\
	= (4n - 1)\left(n(4n-1) - (4n - 2)\cdot 2n + (2n - 1)(4n - 3)\right) \\
	= (n-1)(4n-1)(4n-3)\text{,}
    \end{multline*}
    we get the contradiction
    \begin{multline*}
    	\alpha_1 = (n-1)(4n-1)(4n-3) - \sum_{j=4}^{4n-2}j\binom{j - 2}{2}\alpha_j - (4n-1)\binom{4n-3}{2} \\
	= -(n-1)(4n-1)(4n-3) - \sum_{j=4}^{4n-2}j\binom{j - 2}{2}\alpha_j < 0\text{.}
    \end{multline*}
\end{proof}

\begin{remark}
    Alternatively, Theorem~\ref{thm:family_t_3} can be shown as follows.
    According to \cite[Th.~5.6]{Colbourn-Mathon-SteinerSystems-2007}, the existence of a $3$-$(v,k,1)$ design implies $\binom{v}{3} \geq \frac{v}{k}\binom{v-1}{1}\binom{k}{3}$.
    In our case, this yields the contradiction
    \[
        \frac{2}{3}n(2n-1)(4n-3)(4n-1)(4n^2 - 3n + 1) \geq \frac{2}{3}(2n-1)^2(4n-3)(4n-1)(4n^2-3n+1)\text{.}
    \]
\end{remark}

\section{Intersection structure of a $q$-analog of the Fano plane}
It is a notorious open problem if for any prime power $q$, a \emph{$q$-analog of the Fano plane} exists, which is a Steiner system with admissible parameters $2$-$(7,3,1)_q$.
In this section, we compute the intersection vector distribution of such a Steiner system and discuss its implications.
Thereby, $\Phi_n\in\Z[q]$ will denote the $n$-th cyclotomic polynomial in $q$.
Since all Gaussian binomial coefficients are a product of cyclotomic polynomials, they often allow a compact representation of the arising polynomials in factorized form.

In the following, let $D$ be a $2$-$(7,3,1)_q$ subspace design.
We have
\begin{align*}
	\lambda_0 & = q^8 + q^6 + q^5 + q^4 + q^3 + q^2 + 1 = \Phi_6\Phi_7 \text{,} \\
	\lambda_1 & = q^4 + q^2 + 1 = \Phi_2\Phi_4\text{,} \\
	\lambda_2 & = 1\text{.}
\end{align*}

As a showcase, for $s = 4$ the Köhler equations yield
\begin{align*}
\alpha_0 & = (q^8 - q^7 + q^3) - q^3\alpha_3\text{,} \\
\alpha_1 & = (q^7 + q^6 + q^5 - q^3 - q^2 - q) + (q^3 + q^2 + q)\alpha_3\text{,} \\
\alpha_2 & = (q^4 + q^3 + 2q^2 + q + 1) - (q^2 + q + 1)\alpha_3\text{.}
\end{align*}
Since $\lambda = 1$, $S\in\gaussm{V}{4}$ can contain at most $1$ block, implying $\alpha_3\in\{0,1\}$.
Thus, the two possible intersection vectors are
\[
	(q^8 - q^7 + q^3,\ q^7 + q^6 + q^5 - q^3 - q^2 - q,\ q^4 + q^3 + 2q^2 + q + 1,\ 0)
\]
and
\[
	(q^8 - q^7,\ q^7 + q^6 + q^5,\ q^4 + q^3 + q^2,\ 1)\text{.}
\]
Let $a_i$ ($i\in\{0,1\}$) be the number of $S\in\gaussm{V}{4}$ of the first and the second intersection vectors, respectively.
Double counting the flags $(B,S)\in D\times \gaussm{V}{4}$ with $B < D$ yields
\[
	a_1 = \#D \cdot \gauss{4}{1}{q} = \Phi_2\Phi_4\Phi_6\Phi_7
\]
and thus
\[
	a_0 = \gauss{7}{4}{q} - a_1 = q^4\Phi_6\Phi_7\text{.}
\]

For $s = 3$, the two possible intersection vectors and their frequencies are computed similarly.
For each $s\in\{0,1,2,5,6,7\}$, the intersection vector is uniquely determined by Lemma~\ref{lma:alpha_unique}.
The result is shown in Table~\ref{tbl:fano_intersection_vectors_q}.
For the important special cases $q = 2$ and $q = 3$, the evaluated numbers are shown in Table~\ref{tbl:fano_intersection_vectors_2} and Table~\ref{tbl:fano_intersection_vectors_3}, respectively.
For $s = 3$, we denote the two possible types of subspaces $S$ by $3_0$ (those with $\alpha_3(S) = 0$) and $3_1$ (the blocks with $\alpha_3(S) = 1$).
Similarly, the two different types of blocks of dimension $s = 4$ will be denoted by $4_0$ and $4_1$.

\begin{table}
\centering
	$\begin{array}{llllll}
		\#\text{subspaces } S & s & \alpha_0(S) & \alpha_1(S) & \alpha_2(S) & \alpha_3(S) \\
		\hline
		1                    & 7       & 0       & 0       & 0       & \Phi_6\Phi_7 \\
		\Phi_7               & 6       & 0       & 0       & q^4 \Phi_3\Phi_6 & \Phi_2\Phi_4\Phi_6 \\
		\Phi_3\Phi_6\Phi_7   & 5       & 0       & q^8     & q^3\Phi_2\Phi_4 & \Phi_4 \\
		\Phi_2\Phi_4\Phi_6\Phi_7       & 4       & q^7\Phi_1 & q^5\Phi_3 & q^2\Phi_3 & 1 \\
		q^4\Phi_6\Phi_7      & 4       & q^3(q^5 - q^4 + 1) & q\Phi_1\Phi_2\Phi_3\Phi_4 & \Phi_3\Phi_4 & 0 \\
		\Phi_6\Phi_7 & 3 & q^4\Phi_4\Phi_2\Phi_1 & q^2\Phi_3\Phi_4 & 0 & 1  \\
	        q\Phi_2\Phi_4\Phi_6\Phi_7 & 3 & q^3 (q^5 - q + 1) & q(q^3 + q - 1)\Phi_3 & \Phi_3 & 0 \\

	        \Phi_3\Phi_6\Phi_7 & 2 & q^6\Phi_4 & q^2\Phi_2\Phi_4 & 1 & 0\\
		\Phi_7 & 1 & q^3 \Phi_2\Phi_4\Phi_6 & \Phi_3\Phi_6 & 0 & 0 \\
		1      & 0 & \Phi_6\Phi_7 & 0 & 0 & 0
	\end{array}$
	\caption{Intersection vector distribution of a $2$-$(7,3,1)_q$ design}
	\label{tbl:fano_intersection_vectors_q}
\end{table}

\begin{table}
\centering
	$\begin{array}{llllll}
		\#\text{subspaces } S & s & \alpha_0(S) & \alpha_1(S) & \alpha_2(S) & \alpha_3(S) \\
		\hline
		1                     & 7       & 0       & 0       & 0       & 381     \\
		127                   & 6       & 0       & 0       & 336     & 45      \\
		2667                  & 5       & 0       & 256     & 120     & 5       \\
		5715                  & 4       & 128     & 224     & 28      & 1       \\
		6096                  & 4       & 136     & 210     & 35      & 0       \\
		381                   & 3       & 240     & 140     & 0       & 1       \\
		11430                 & 3       & 248     & 126     & 7       & 0       \\
		2667                  & 2       & 320     & 60      & 1       & 0       \\
		127                   & 1       & 360     & 21      & 0       & 0       \\
		1                     & 0       & 381     & 0       & 0       & 0
	\end{array}$
	\caption{Intersection vector distribution of a $2$-$(7,3,1)_2$ design}
	\label{tbl:fano_intersection_vectors_2}
\end{table}

\begin{table}
\centering
	$\begin{array}{llllll}
		\#\text{subspaces } S & s & \alpha_0(S) & \alpha_1(S) & \alpha_2(S) & \alpha_3(S) \\
		\hline
		1                     & 7       & 0       & 0       & 0       & 7651    \\
		1093                  & 6       & 0       & 0       & 7371    & 280     \\
		99463                 & 5       & 0       & 6561    & 1080    & 10      \\
		306040                & 4       & 4374    & 3159    & 117     & 1       \\
		619731                & 4       & 4401    & 3120    & 130     & 0       \\
		7651                  & 3       & 6480    & 1170    & 0       & 1       \\
		918120                & 3       & 6507    & 1131    & 13      & 0       \\
		99463                 & 2       & 7290    & 360     & 1       & 0       \\
		1093                  & 1       & 7560    & 91      & 0       & 0       \\
		1                     & 0       & 7651    & 0       & 0       & 0
	\end{array}$
	\caption{Intersection vector distribution of a $2$-$(7,3,1)_3$ design}
	\label{tbl:fano_intersection_vectors_3}
\end{table}

\begin{theorem}
	Let $q$ be a prime power.
	The existence of a $2$-$(7,3,1)_q$ subspace design implies the existence of a $2$-$(7,3,q^4)_q$ design.
\end{theorem}

\begin{proof}
Let $S$ be a $5$-space in $V$.
By the unique intersection vector for $s = 5$, there are $\Phi_4 = q^2 + 1$ blocks contained in $S$.
A $4$-space $W \leq S$ is of type $4_1$ if and only if it contains one of those blocks.
For each such block $B$, there are $\gauss{5-3}{4 - 3}{q} = \gauss{2}{1}{q} = q + 1$ intermediate $4$-spaces $W$ with $B \leq W \leq S$.
This gives us the number of spaces of type $4_1$ in $S$ as $(q^2 + 1)(q + 1) = q^3 + q^2 + q + 1$.
Therefore, the number of spaces of type $4_0$ in $S$ is $\gauss{5}{4}{q} - (q^3 + q^2 + q + 1) = q^4$.
This shows that $\{S^\perp \mid S \in \gaussm{V}{4} \text{ of type } 4_0\}$ forms a $2$-$(7,3,q^4)_{q}$ design.
\end{proof}

\begin{remark}
The blocks of the original $2$-$(7,3,1)_q$ design are given by the spaces of type $3_1$, and the above proof shows that after dualization, the spaces of type $4_0$ form the blocks of a $2$-$(7,3,q^4)_q$ design.
Similarly, the spaces of type $3_0$ are the blocks of a $2$-$(7,3,q^4 + q^3 + q^2 + q)_q$ design and after dualization, the spaces of type $4_1$ are the blocks of a $2$-$(7,3,q^3 + q^2 + q + 1)_q$ design.
However, these are just the complementary designs of the ones arising from the spaces $3_1$ and $4_0$, respectively.
\end{remark}

The resulting ``intersection structure'' of a $2$-analog of the Fano plane is shown in Figure~\ref{fig:fano_intersection_structure_2}.
We explain by a few examples how to read this figure:
The entry $(128,224,28,1)^{5715}$ on the level $s=4$ means that there are $5715$ subspaces of dimension $s=4$ having the intersection vector $(128,224,28,1)$ (the subspaces of type $4_1$).
It is connected by a line to the intersection vector $(248,126,7,0)$ (type $3_0$) because a subspace of type $4_1$ contains subspaces of type $3_0$.
More precisely, the number $14$ at the $4_1$-end of the line tells us that each subspace of type $4_1$ contains exactly $14$ subspaces of type $3_0$.
Similarly, the number $7$ at the $3_0$-end means that each subspace of type $3_0$ is contained in exactly $7$ subspaces of type $4_1$.

\begin{figure}
	\centering
\begin{tikzpicture}[auto]
  \node at (-4,0) {$s=0$};
  \node (D0) at (2,0) {$(381,0,0,0)^1$};
  \node at (-4,2) {$s=1$};
  \node (D1) at (2,2) {$(360,21,0,0)^{127}$};
  \node at (-4,4) {$s=2$};
  \node (D2) at (2,4) {$(320,60,1,0)^{2667}$};
  \node at (-4,6) {$s=3$};
  \node (D31) at (0,6) {$(240,140,0,1)^{381}$};
  \node (D32) at (4,6) {$(248,126,7,0)^{11430}$};
  \node at (-4,8) {$s=4$};
  \node (D41) at (0,8) {$(128,224,28,1)^{5715}$};
  \node (D42) at (4,8) {$(136,210,35,0)^{6096}$};
  \node at (-4,10) {$s=5$};
  \node (D5) at (2,10) {$(0,256,120,5)^{2667}$};
  \node at (-4,12) {$s=6$};
  \node (D6) at (2,12) {$(0,0,336,45)^{127}$};
  \node at (-4,14) {$s=7$};
  \node (D7) at (2,14) {$(0,0,0,381)^1$};
  \draw (D0) to node[very near start,font=\scriptsize,swap]{$127$} node[very near end,font=\scriptsize]{$1$} (D1);
  \draw (D1) to node[very near start,font=\scriptsize,swap]{$63$} node[very near end,font=\scriptsize]{$3$} (D2);
  \draw (D2) to node[very near start,font=\scriptsize,swap]{$1$} node[very near end,font=\scriptsize]{$7$} (D31);
  \draw (D2) to node[pos=0.35,font=\scriptsize,swap]{$30$} node[pos=0.65,font=\scriptsize]{$7$} (D32);
  \draw (D31) to node[very near start,font=\scriptsize,swap]{$15$} node[very near end,font=\scriptsize]{$1$} (D41);
  \draw (D32) to node[very near start,font=\scriptsize,swap]{$7$} node[very near end,font=\scriptsize]{$14$} (D41);
  \draw (D32) to node[very near start,font=\scriptsize,swap]{$8$} node[very near end,font=\scriptsize]{$15$} (D42);
  \draw (D41) to node[pos=0.35,font=\scriptsize,swap]{$7$} node[pos=0.65,font=\scriptsize]{$15$} (D5);
  \draw (D42) to node[very near start,font=\scriptsize,swap]{$7$} node[very near end,font=\scriptsize]{$16$} (D5);
  \draw (D5) to node[very near start,font=\scriptsize,swap]{$3$} node[very near end,font=\scriptsize]{$63$} (D6);
  \draw (D6) to node[very near start,font=\scriptsize,swap]{$1$} node[very near end,font=\scriptsize]{$127$} (D7);
\end{tikzpicture}
	\caption{Intersection structure of a $2$-$(7,3,1)_2$ design}
	\label{fig:fano_intersection_structure_2}
\end{figure}
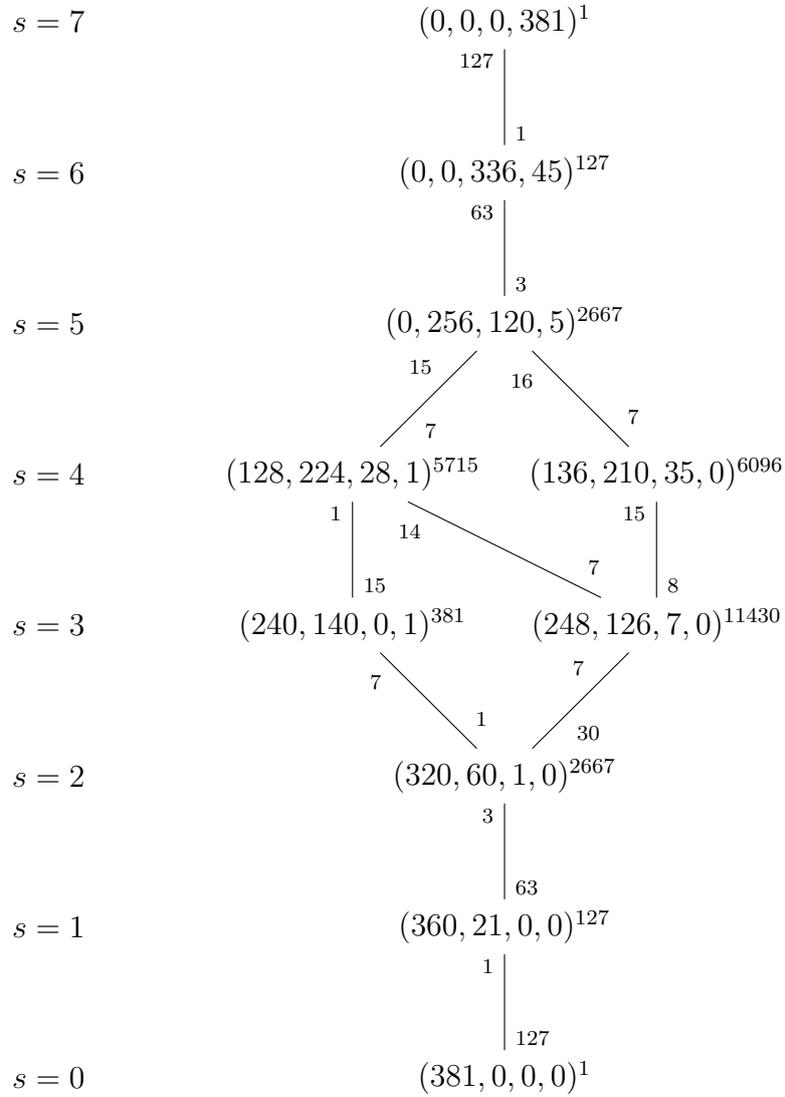

\section*{Acknowledgement}
The authors would like to acknowledge the financial support provided by COST -- \emph{European Cooperation in Science and Technology}. Both authors are members of the Action IC1104 \emph{Random Network Coding and Designs over GF(q)}.
This research was carried out during a 6 week stay of the first author at the University of Zagreb, which was supported by an STSM grant of the COST project.

We would like to express our gratitude to R. Laue for pointing out this problem, giving us hints for references and sharing some very useful thoughts with us.

\printbibliography
\end{document}